\documentclass[reqno,b5paper]{amsart}
\usepackage{amsmath}
\usepackage{amssymb}
\usepackage{amsthm}
\usepackage{enumerate}
\usepackage[mathscr]{eucal}
\usepackage{graphicx}
\newtheorem{theorem}{Theorem}[section]
\newtheorem{lemma}[theorem]{Lemma}
\newtheorem{proposition}[theorem]{Proposition}
\newtheorem{corollary}[theorem]{Corollary}

\newtheorem{remark}{Remark}

\theoremstyle{definition}
\newtheorem{example}[theorem]{Example}

\begin{document}
\title[The Bi-Compact-Open Topology on $C(X)$]{The Bi-Compact-Open Topology on $C(X)$}
\author{Anubha Jindal}
\address{Anubha Jindal: Department Of Mathematics, Indian Institute of Technology Delhi,
New Delhi 110016, India.}
\email{jindalanubha217@gmail.com}

\author{R. A. McCoy}
\address{R. A. McCoy: Department of Mathematics, Virginia Tech, Blacksburg VA 24061-0123, U.S.A.}
\email{mccoy@math.vt.edu}

\author{S. Kundu}
\address{S. Kundu:  Department Of Mathematics, Indian Institute of Technology Delhi,
New Delhi 110016, India.}
\email{skundu@maths.iitd.ac.in}

\begin{abstract} For the set $C(X)$ of real-valued continuous functions on a Tychonoff space $X$, the compact-open topology on $C(X)$ is a ``set-open topology." This paper studies the separation and countability properties of the space $C(X)$ having the topology given by the join of the compact-open topology and an ``open-set topology" called the open-point topology, that was introduced in \cite{AMK}.
\end{abstract}

\maketitle
\section{\textbf{Introduction}}
The set $C(X)$ of all real-valued continuous functions on a Tychonoff space $X$ has a number of natural topologies. One important type of topology on $C(X)$ is the set-open topology, introduced by Arens and Dugunji \cite{AD}, and later studied particularly in \cite{K}, \cite{AO}, \cite{O}, \cite{Osi} and \cite{Velichko}. In the definition of a set-open topology on $C(X)$, we use a certain family of subsets of $X$ and open subsets of $\mathbb{R}$. In \cite{AMK}, by adopting a radically different approach, we have defined two new kinds of topologies on $C(X)$ known as the open-point and bi-point-open topologies. One main reason for adopting such a different approach is to ensure that both $X$ and $\mathbb{R}$ play equally significant roles in the construction of topologies on $C(X)$. This gives a function space where one has more control of the functions in it.

The open-point topology on $C(X)$ has a subbase consisting of sets of the form
$$[U, r]^- = \{f \in C(X) : f^{-1}(r) \cap U\neq\emptyset\},$$
where $U$ is an open subset of $X$ and $r \in \mathbb{R}$. The open-point topology on $C(X)$ is denoted by $h$ and
the space $C(X)$ equipped with the open-point topology $h$ is denoted by $C_h(X)$. The term $``h"$ comes from the word
``horizontal" because, for example, a subbasic open set in $C_h(\mathbb{R})$ can be viewed as the set of functions in $C(\mathbb{R})$ whose
graphs pass through some given horizontal open segment in $\mathbb{R}\times \mathbb{R}$, as opposed to a subbasic open set
in $C_p(\mathbb{R})$ which consists of the set of functions in $C(\mathbb{R})$ whose graphs pass through some given vertical
open segment in $\mathbb{R} \times \mathbb{R}$.

Among the set-open topologies on $C(X)$, the compact open topology $k$ occupies an eminent place. This topology, since its introduction in 1945 by Fox in \cite{f}, has been extensively studied by many people from the view point of topology as well as from the view point of its dual used in analysis. This motivates one to study the open-compact topology on $C(X)$, that is, the topology consisting of subbasic open sets of the form
 $$[U, B]^- = \{f \in C(X) : f^{-1}(B) \cap U\neq\emptyset\},$$
where $U$ is an open subset of $X$ and $B$ is compact set in $\mathbb{R}$. But one can prove that both the collections $$\{[U,r]^-: U \mbox{ is open in }X \mbox{ and }r\in\mathbb{R}\}$$ and $$\{[U,B]^-: U \mbox{ is open in }X \mbox{ and }B \mbox{ is compact in }\mathbb{R}\}$$ generate the same topology on $C(X)$, and that is the open-point topology ``h''.

 As we have defined in \cite{AMK}, the bi-point-open topology on $C(X)$ is the join of the point-open topology $p$ and the open-point topology $h$. In other words, it is the topology having subbasic open sets of both kinds: $[x,V]^{+}=\{f\in C(X): f(x)\in V\}$ and $[U,r]^{-}$, where $x\in X$ and $V$ is an open subset of $\mathbb{R}$, while $U$ is an open subset of $X$ and $r \in \mathbb{R}$. The bi-point-open topology on the space $C(X)$ is denoted by $ph$ and the space $C(X)$ equipped with the bi-point-open topology $ph$ is denoted by $C_{ph}(X)$. One can also view the bi-point-open topology on $C(X)$ as the weak topology on $C(X)$ generated by the identity maps $id_1:C(X)\rightarrow C_p(X)$ and $id_2:C(X)\rightarrow C_h(X)$.

Similarly, we can define the bi-compact-open topology on $C(X)$ as the join of the compact-open topology $k$ and the open-point topology $h$. In other words, it is the topology having subbasic open sets of both kinds: $[A,V]^{+}=\{f \in C(X) : f(A) \subseteq V \}$ and $[U,r]^{-}$, where $A$ is a compact subset of $X$ and $V$ is open in $\mathbb{R}$, while $U$ is an open subset of $X$ and $r \in \mathbb{R}$. The bi-compact-open topology on $C(X)$ is denoted by $kh$ and the space $C(X)$ equipped with the bi-compact-open topology $kh$ is denoted by $C_{kh}(X)$. One can also view the bi-compact-open topology on $C(X)$ as the weak topology on $C(X)$ generated by the identity maps $id_1:C(X)\rightarrow C_k(X)$ and $id_2:C(X)\rightarrow C_h(X)$.

In \cite{AMK} and \cite{AMK2}, we have shown that the space, $C_{ph}(X)$ has some nice properties. For example, the pseudocharacter of $C_{ph}(X)$ is equal to the density of $X$; that is, $\psi(C_{ph}(X))=d(X)$. But most of the properties of the spaces $C_h(X)$ and $C_{ph}(X)$ come under strong restriction on $X$. For example, the space $C_{ph}(X)$ is Tychonoff if and only if the set of isolated points in $X$ is $G_\delta$-dense. The properties of the space $C_{kh}(X)$ seem to be less restricted than the same properties of the spaces $C_h(X)$ and $C_{ph}(X)$. So in the authors' opinion, $C_{kh}(X)$ is more interesting and tends to have stronger topological properties.

In Section 2, we  give some bases and study the separation properties of the space $C_{kh}(X)$. It is shown that the space $C_{kh}(X)$ is completely regular if and only if the set of locally compact points in $X$ is $G_\delta$-dense in $X$. In Section 3, the submetrizability and first countability of the space $C_{kh}(X)$ are characterized. In the last section, we discuss the separability of the space $C_{kh}(X)$.

 Throughout this paper the following conventions are used. The symbols $\mathbb{R}$, $\mathbb{Q}$, $\mathbb{Z}$ and $\mathbb{N}$ denote the space of real numbers, rational numbers, integers and natural numbers, respectively. For a space $X$ the symbol $X^0$ denotes the set of all isolated points in $X$, $|X|$ denotes the cardinality of the space $X$, $\overline{A}$ denotes the closure of $A$ in $X$, $A^c$ denotes the complement of $A$ in $X$ and $0_X$ denotes the constant zero-function in $C(X)$. Also for any two topological spaces $X$ and $Y$ that have the same underlying set, the three expressions $X=Y$, $X \leq Y$ and $X < Y$ mean that, respectively, the topology of $X$ is same as topology of $Y$, the topology of $X$ is weaker than or equal to topology of $Y$ and the topology of $X$ is  strictly weaker than the topology of $Y$. For other basic topological notions, refer to \cite{re}.
\section{Preliminaries}
In this section, we study the separation axioms of the space $C_{kh}(X)$. We first give some bases for the space $C_{kh}(X)$ that are useful in establishing properties for this space. Here $\mathcal{U}$ is some given base for $X$ and $\mathcal{V}$ is some given countable base for $\mathbb{R}$ consisting of bounded open intervals. Now Proposition 2.1 in \cite{AMK} implies the following proposition.

\begin{proposition}\label{base for Ckh(X)}
The space $C_{kh}(X)$ has a base consisting of sets of the form $[A_1,V_1]^+\cap\ldots\cap[A_m,V_m]^+\cap [U_1, r_1]^-\cap\ldots\cap[U_n, r_n]^-$, where $m,n\in\mathbb{N}$, $A_i$ is a compact subset of $X$, $V_i\in \mathcal{V}$, $U_j\in \mathcal{U}$, $r_j\in \mathbb{R}$, whenever $1\leq i\leq m$ and $1\leq j\leq n$, and $\overline{U_i}\cap\overline{U_j}=\emptyset$ for $i\neq j$.
\end{proposition}

\begin{proposition}\label{pibase for Ckh(X)}
The space $C_{kh}(X)$ has a $\pi$-base of sets of the form $[A_1,W_1]^+\cap\ldots\cap[A_m,W_m]^+ \cap [U_1, r_1]^-\cap\ldots\cap[U_n, r_n]^-$, where $m,n\in\mathbb{N}$, $A_i$ is a compact subset of $X$, $W_i\in \mathcal{V}$, $U_j\in \mathcal{U}$, $r_j\in \mathbb{R}$, whenever $1\leq i\leq m$ and $1\leq j\leq n$; and for each $j = 1,\ldots,n$, either $U_j\cap A_i=\emptyset$ for all $i=1,\ldots,m$, or $U_j\subseteq\cup\{A_i : r_j\in W_i\}$ and $U_j\cap \cup\{A_i : r_j \notin W_i\}=\emptyset$, and $\overline{U_i}\cap \overline{U_j}=\emptyset$ for $i\neq j$.
\end{proposition}
\begin{proof}
Let $B=[A_1,W_1]^+\cap\ldots\cap[A_m,W_m]^+ \cap [U_1, r_1]^-\cap\ldots\cap[U_n, r_n]^-$ be any nonempty basic open set in $C_{kh}(X)$, where $A_i$ is a compact subset in $X$ and $W_i\in\mathcal{V}$ for each $i \in \{1,\ldots,m\}$; and $U_j$ is an open set in $X$ and $r_j\in \mathbb{R}$ for each $j \in \{1,\ldots,n\}$; and $\overline{U_i}\cap \overline{U_j}=\emptyset$ for $i\neq j$. So for $f\in B$, there exists $x_j\in U_j$ such that $f(x_j)=r_j$ for each $j \in \{1,\ldots,n\}$ and $f(A_i)\subseteq W_i$ for each $i \in \{1,\ldots,m\}$. Now for each $j=1,\ldots,n$, we have either $U_j\subseteq A_1\cup\ldots\cup A_m$ or $U_j\setminus (A_1\cup\ldots\cup A_m)\neq\emptyset$. Without loss of generality, we can assume that $U_j\subseteq A_1\cup\ldots\cup A_m$ for $1\leq j\leq k$, where $k\leq n$ and $U_j\setminus (A_1\cup\ldots\cup A_m)\neq\emptyset$ for $k+1\leq j\leq n$. Now for $1\leq j\leq k$, $U_j^{'}=U_j\setminus\cup\{A_i : r_j \notin W_i\}$ is a nonempty open subset of $X$ contained in $\cup\{A_i : r_j\in W_i\}$. Also $x_j\in U_j^{'}$ and $U_j^{'}\cap\cup\{A_i : r_j \notin W_i\}=\emptyset$. For $k+1\leq j\leq n$, let $y_j\in U_j\setminus (A_1\cup\ldots\cup A_m)=U_j^{'}$. Since $U_i\cap U_j=\emptyset$, we have $\{y_{k+1},\ldots,y_n\}$ is a collection of distinct points in $X$. Now define $g\in C(X)$ such that $g(x)=f(x)$ for each $x\in A_1\cup\ldots\cup A_m$ and $g(y_j)=r_j$ for $k+1\leq j\leq n$. Then $g\in [A_1,W_1]^+\cap\ldots\cap[A_m,W_m]^+ \cap [U_1^{'}, r_1]^-\cap\ldots\cap[U_n^{'}, r_n]^-=B'$ and $B'\subseteq B$.
\end{proof}

The space $C_{kh}(X)$ has finer topology than the Hausdorff space $C_k(X)$, so that $C_{kh}(X)$ is always a Hausdorff space. However, $C_{kh}(X)$ need not be completely regular. The next theorem characterizes those spaces $X$ for which $C_{kh}(X)$ is completely regular. For this theorem we need two lemmas, the first of which concerns the set of points of a space $X$ that have compact neighborhoods. Let us call such a point a locally compact point of $X$, and let $lc(X)$ denote the set of such points. Also recall that a subset $A$ of a space $X$ is called $G_{\delta}$-dense in $X$ if every nonempty $G_{\delta}$-set in $X$ intersects $A$.

\begin{lemma}\label{locombaselemmaforopenpoint}
If $lc(X)$ is $G_\delta$-dense in $X$, then $C_h(X)$ has a base consisting of sets of the form $[U_1, r_1]^-\cap\ldots\cap[U_n, r_n]^-$, where each open set $U_i$ has compact closure and $\overline{U_i}\cap \overline{U_j}=\emptyset$ for $i\neq j$.
\end{lemma}
\begin{proof}
Let $G=[V_1, r_1]^-\cap\ldots\cap[V_n, r_n]^-$ be any nonempty basic open set in $C_h(X)$, where $V_i$ is open in $X$, $r_i$ is in $\mathbb{R}$ and $\overline{V_i}\cap \overline{V_j}=\emptyset$ for each $i\neq j$. If $f\in G$, then for each $i \in \{1,\ldots,n\}$, we have $f^{-1}(r_i)\cap V_i\neq\emptyset$. Since $lc(X)$ is $G_\delta$-dense in $X$, for each $i \in \{1,\ldots,n\}$, there exist $x_i$ in $lc(X)$ and an open set $U_i$ having compact closure such that $x_i\in f^{-1}(r_i)\cap V_i$ and $x_i\in U_i\subseteq\overline{U_i}\subseteq V_i$. Since $\overline{V_i}\cap \overline{V_j}=\emptyset$ for $i\neq j$, $\overline{U_i}\cap \overline{U_j}=\emptyset$ for $i\neq j$. Hence $f\in [U_1, r_1]^-\cap\ldots\cap[U_n, r_n]^-\subseteq G$.
\end{proof}

\begin{lemma}\label{zerosetAndContinuousFunction}
 If $A$ is a nonempty $G_\delta$-set in $X$, then there exists $f\in C(X)$ such that $f\geq 0$, $f^{-1}(0)\neq \emptyset$ and $f^{-1}(0)\subseteq A$.
\end{lemma}
\begin{proof}
See 3.11.(b) in \cite{GJ}.
\end{proof}

\begin{theorem}\label{Tychonoffness of C_kh(X)}
For any space $X$, the following are equivalent.
\begin{enumerate}
  \item [$(a)$] $C_{kh}(X)$ is completely regular.
  \item [$(b)$] $C_{kh}(X)$ is regular.
  \item [$(c)$] $lc(X)$ is $G_\delta$-dense in $X$.
\end{enumerate}
\end{theorem}
\begin{proof}

\noindent $(a)\Rightarrow (b)$. Every completely regular space is regular.

\noindent $(b)\Rightarrow (c)$. Suppose that $lc(X)$ is not $G_\delta$-dense in $X$. So there exists a nonempty $G_\delta$-set $A$ in $X$ such that $lc(X)\cap A=\emptyset$. By Lemma \ref{zerosetAndContinuousFunction} there exists $f\in C(X)$ such that $\emptyset\neq f^{-1}(0)\subseteq A$ and $f\geq 0$. As $lc(X)\cap A=\emptyset$, the set $f^{-1}(0)$ contains no locally compact point. Since $f^{-1}(0)\cap A\neq \emptyset$, $F=C_{kh}(X)\setminus [X,0]^-$ is a closed subset of $C_{kh}(X)$ which does not contain $f$. Let $B=[A_1,V_1]^+\cap\ldots\cap[A_m,V_m]^+\cap[U_1, r_1]^-\cap\ldots\cap[U_n,r_n]^-$ be any basic open set and $W$ be any open set in $C_{kh}(X)$ such that $f\in B$ and $F\subseteq W$, where $\overline{U_i}\cap\overline{U_j}=\emptyset$ for $i\neq j$. Since $f\geq 0$ and $f\in B$, for each $j \in \{1,\ldots,n\}$, there exists $x_j\in U_j$ such that $f(x_j)=r_j\geq 0$ . Let $I=\{i\in \{1,\ldots,n\}: r_i=0\}$ and $J=\{1,\ldots, n\}\setminus I$. When $i\in I$, $x_i$ is not a locally compact point of $X$. Now we show that $W\cap B\neq\emptyset$, and thus $C_{kh}(X)$ is not regular.

  Let $a_i=\min{f(A_i)}$ for $1\leq i\leq m$. If $a_i=0$ for some $i\in\{1,\ldots,m\}$, then $0\in V_i$. For each such $i$, since $V_i$ is open, there exists $\delta_i>0$ such that $[0,\delta_i)\subseteq V_i$. Then take  $S=\{\frac{\delta_i}{2}:a_i=0\}\cup\{\frac{a_i}{2}:a_i\neq 0\}\cup\{\frac{r_j}{2}:j\in J\}$ and $\epsilon=\min\{s:s\in S\}$.
Now define $g:X\rightarrow \mathbb{R}$ by $g(x)=\max\{f(x),\epsilon\}$. Clearly $g\in C(X)$ such that $g>0$, $g\in [A_1,V_1]^+\cap\ldots\cap[A_m,V_m]^+$ and $g(x_j)=r_j$ for each $j\in J$. Hence $g\in F\subseteq W$.

 Let $B'=[C_{1}^{'}, W_{1}^{'}]^+\cap\ldots\cap[C_{l}^{'},W_{l}^{'}]^+\cap[U_{1}^{'}, s_1]^-\cap \ldots \cap [U_{q}^{'}, s_q]^-$ be a basic neighborhood of $g$ in $C_{kh}(X)$ with $B'\subseteq W$, and let $K=\{1,\ldots,q\}$.
 For each $k\in K$, choose $z_k\in U_{k}^{'}$ such that $g(z_k)=s_k$. Take $A=\{x_j:j\in J\}\cup\{z_k:k\in K\}\cup A_1\cup \ldots\cup A_m\cup C_{1}^{'}\cup\ldots\cup C_{l}^{'}$.
 For each $i\in I$, $U_i\nsubseteq A$, since $x_i$ is not a locally compact point of $X$. For each $i\in I$, let $x_{i}^{'}\in U_i\setminus A$.
 Define continuous $h:A\cup \{x_{i}^{'}:i\in I\}\rightarrow \mathbb{R}$ such that $h(x_{i}^{'})=0$ for each $i\in I$ and $h(x)=g(x)$ for each $x\in A$. Since $A\cup \{x_{i}^{'}:i\in I\}$ is compact, we can extend $h$ continuously on $X$. Then $h\in B\cap B'\subseteq B\cap W$.

 \noindent  $(c)\Rightarrow (a)$. It is enough to prove the result for subbasic open sets as the maximum or minimum of a finite number of continuous functions is continuous. So let $[U,r]^-$ be a subbasic open set in $C_{kh}(X)$ such that $\overline{U}$ is a compact subset of $X$ and $f\in [U, r]^-$. Let $x\in f^{-1}(r) \cap U$. Since $X$ is a completely regular space, there exists a continuous function $\phi : X \rightarrow [0,1]$ such that $\phi(x) = 0$ and $\phi(y) = 1$ for $y\notin U$. For each $g\in C_{kh}(X)$, define
\[\Phi(g) =\left\{\begin{array}{cc}
                  \inf\phi(g^{-1}(r)\cap U)   & g^{-1}(r)\cap U\neq \emptyset  \\
                  1 & g^{-1}(r)\cap U =\emptyset.
                \end{array}\right.
   \]
Clearly, $\Phi(f) = 0$ and $\Phi(g) = 1$ for $g\notin [U, r]^-$. Now we show that the function $\Phi:C_{kh}(X)\rightarrow [0,1]$ is continuous.

Let $g\in C_{kh}(X)$, let $t=\Phi(g)$ and let $\epsilon > 0$. Suppose first that $g^{-1}(r)\cap U\neq \emptyset$ and that $t>0$. In this case, we can assume that $t-\epsilon > 0$. Since $\overline{U}$ is compact, the set $G=[U\cap \phi^{-1}(t-\frac{\epsilon}{2},t+\frac{\epsilon}{2}),r]^-\cap[\overline{U}\cap \phi^{-1}([0,t-\frac{\epsilon}{2}]),\{r\}^c]^+$ is an open set in $C_{kh}(X)$. Now we show that $g\in G$ and  $\Phi(G)\subseteq (t-\epsilon,t+\epsilon)$. Since $\Phi(g)=t$, there exists $y\in g^{-1}(r)\cap U$ such that $t\leq \phi(y)<t+\frac{\epsilon}{2}$. Since $\phi(s) = 1$ for $s\notin U$, we must have $\phi(z)\geq t$ for all $z\in g^{-1}(r)\cap \overline{U}$. So, we have $g^{-1}(r)\cap\overline{U}\cap \phi^{-1}[0,t-\epsilon]=\emptyset$. Thus $g\in G$. For each $h\in G$, there exists $y\in U$ such that $h(y)=r$ and $\phi(y)\in(t-\frac{\epsilon}{2},t+\frac{\epsilon}{2})$, and $h(z)\neq r$ for all $z\in \overline{U}\cap \phi^{-1}[0,t-\frac{\epsilon}{2}]$. Thus $t-\frac{\epsilon}{2}\leq\inf\phi(h^{-1}(r)\cap U)\leq t+\frac{\epsilon}{2}$. Hence  $\Phi(G)\subseteq (t-\epsilon,t+\epsilon)$ for the case that $t>0$. For the case that $t=0$, take $G=[U\cap \phi^{-1}(-\frac{\epsilon}{2},\frac{\epsilon}{2}),r]^-$. Then the same proof will work as that for the case $t>0$.

On the other hand, if $g^{-1}(r)\cap U= \emptyset$, then $\Phi(g)=1$. Choose any $0<\epsilon<1$. Consider the neighborhood $(1-\epsilon,1]$ of $\Phi(g)$. Take $B=[\overline{U}\cap \phi^{-1}[0,1-\frac{\epsilon}{2}], \{r\}^c ]^+$. Now we prove that $g\in B$ and $\Phi(B)\subseteq (1-\epsilon,1]$. Since for $y\in \overline{U}\setminus U$, $\phi(y)=1$ and $g^{-1}(r)\cap U= \emptyset$, we have $\overline{U}\cap g^{-1}(r)\cap \phi^{-1}[0,1-\frac{\epsilon}{2}]=\emptyset$. Thus $g\in B$. Note that if $h\in B$, then $h(z)\neq r$ for all $z\in \overline{U}\cap \phi^{-1}[0,1-\frac{\epsilon}{2}]$. Thus for such an $h$, $1-\frac{\epsilon}{2}\leq\inf\phi(h^{-1}(r)\cap U)\leq 1$. Then this implies that $\Phi(B)\subseteq (1-\epsilon,1]$. So $\Phi$ is continuous.

Now let us consider a subbasic open set of the form $[A,V]^+$ in $C_{kh}(X)$ and $f\in [A,V]^+$. Since $C_k(X)$ is a completely regular space, there exists a continuous function $\Psi : C_{kh}(X)\rightarrow [0,1]$ such that $\Psi(f) = 0$ and $\Psi(g) = 1$ for $g\notin [A, V]^+$. By using Lemma \ref{locombaselemmaforopenpoint} and the continuity of $\Phi$ and $\Psi$, one can prove that $C_{kh}(X)$ is a completely regular space.
\end{proof}

\begin{remark}
In \cite{AMK}, it is shown that the space $C_{ph}(X)$  is completely regular if and only if $X^0$ is $G_\delta$-dense in $X$ if and only if $C_{h}(X)$ is completely regular. Since $X^0\subseteq lc(X)$, complete regularity of the space $C_{ph}(X)$ implies the complete regularity of the space $C_{kh}(X)$. But the converse need not be true. For example Theorem \ref{Tychonoffness of C_kh(X)} implies that the space $C_{kh}(\mathbb{R})$ is completely regular, while the spaces $C_h(\mathbb{R})$ and $C_{ph}(\mathbb{R})$ are not completely regular.
\end{remark}

We end this section by examining when the spaces $C_{ph}(X)$ and $C_{kh}(X)$ have the same topology; that is, when $C_{ph}(X)=C_{kh}(X)$. If $C_{p}(X)=C_{k}(X)$, then obviously $C_{ph}(X)=C_{kh}(X)$. But if $C_{ph}(X)=C_{kh}(X)$, then it is not clear whether $C_{p}(X)=C_{k}(X)$. In the next theorem, we show that it is in fact true.
\begin{theorem}
For a space $X$, the following are equivalent.
\begin{enumerate}
  \item [$(a)$] $C_{p}(X)=C_{k}(X)$
  \item [$(b)$] $C_{ph}(X)=C_{kh}(X)$
  \item [$(c)$] Every compact subset of $X$ is finite.
\end{enumerate}
\end{theorem}
\begin{proof}
\noindent $(a)\Rightarrow (b)$. It is immediate.

\noindent $(b)\Rightarrow (c)$. Suppose that $A$ is an infinite compact subset of $X$. Let $[A,V]^+$ be an open set in $C_{kh}(X)$, where $V$ is an open proper subset of $\mathbb{R}$. Let $f\in [A,V]^+$. Since $C_{ph}(X)= C_{kh}(X)$, there exists an open set $B=[y_1,V_1]^+\cap \ldots \cap [y_m,V_m]^+\cap [U_1,r_1]^-\cap\ldots\cap [U_n,r_n]^-$ in $C_{ph}(X)$ such that $f\in B\subseteq [A,V]^+$, where $y_i\in X$ and $V_i$ is an open subset of $\mathbb{R}$ for each $i \in \{1,\ldots,m\}$; also $U_j$ is an open subset of $X$ and $r_j$ is an element of $\mathbb{R}$ for each $j \in \{1,2,\ldots,n\}$. Since $f\in B$, there exists $x_j\in U_j$ such that $f(x_j)=r_j$ for each $j \in \{1,\ldots,n\}$ and $f(y_i)\in V_i$ for each $i \in \{1,\ldots,m\}$.

Choose $x_0\in A\setminus(\{y_1,\ldots,y_m\}\cup\{x_1,\ldots,x_n\})$ and $r\in \mathbb{R}\setminus V$. Take $S=\{y_i:1\leq i\leq m\}\cup\{x_j:1\leq j\leq n\}\cup \{x_0\}$. Since $S$ is a compact subset of $X$, there exists $g\in C(X)$ such that $g(x_0)=r$; and $g(y_i)\in V_i$ for each $i \in \{1,\ldots,m\}$ and $g(x_j)=r_j$ for each $j \in \{1,\ldots,n\}$. Thus $g\in B$ but $g\notin [A,V]^+$.

\noindent $(c)\Rightarrow (a)$. It is well-known.
\end{proof}

\begin{theorem}
For a space $X$, the space $C_h(X)$ can never be weaker than the space $C_k(X)$.
\end{theorem}
\begin{proof}
Now $[U,0]^-$ is an open set in $C_h(X)$ containing the constant zero-function $0_X$. Let $B=\langle0_X,A,\epsilon\rangle=\{f\in C(X):|f(x)|<\epsilon \mbox{ for all } x\in A\}$ be any basic neighborhood of $0_X$ in $C_k(X)$, where $A$ is a compact subset of $X$ and $\epsilon>0$. Take $f\in C(X)$ such that $f(x)=\frac{\epsilon}{2}$ for all $x\in X$. Then $f\in B$ but $f\notin [U,0]^-$. So $[U,0]^-$ cannot be open in $C_k(X)$.
\end{proof}
So the space $C_k(X)$ is either strictly weaker than the space $C_h(X)$ or there is no relation. In our next theorem we show that the space $C_k(X)$ is strictly weaker than $C_h(X)$ if and only if $X$ is a discrete space.
\begin{theorem}\label{comparison of topologiesCkh(X)}
For a space $X$, the following are equivalent.
\begin{enumerate}
  \item [$(a)$] $C_k(X)< C_h(X)$.
  \item [$(b)$] $C_{kh}(X)= C_h(X)$.
  \item [$(c)$] $X$ is a discrete space.
\end{enumerate}
\end{theorem}
\begin{proof}
\noindent $(a)\Rightarrow (b)$. This follows from $C_h(X)\leq C_{kh}(X)$.

\noindent $(b)\Rightarrow (c)$. Suppose that $X$ is not discrete. Then $X$ has some non-isolated point $x_0$. Now $[x_0,(-1,1)]^+$ is an open neighborhood of the constant zero-function $0_X$  in $C_{kh}(X)$. To show that $[x_0,(-1,1)]^+$ is not a neighborhood of $0_X$ in $C_h(X)$ let $$B=[U_1,0]^-\cap\ldots\cap[U_n,0]^-$$ be any basic neighborhood of $0_X$ in $C_h(X)$. Since $x_0$ is non-isolated, there exists $x_i\in U_i\setminus\{x_0\}$ for each $1\leq i\leq n$. Let $g\in C(X)$ be such that $g(x_0)=1$ and $g(x)=0$ for all $x\in \{x_1,\ldots,x_n\}$. Then $g\in B$ but $g\notin [x_0,(-1,1)]^+$. So $[x_0,(-1,1)]^+$ cannot be open in $C_h(X)$.

\noindent $(c)\Rightarrow (a)$. Suppose that $X$ is discrete. Take any subbasic open set $[A,V]^+=\{f\in C(X):f(A)\subseteq V\}$ in $C_k(X)$, where $A$ is a compact set in $X$ and $V$ is an open set in $\mathbb{R}$. Since $X$ is discrete, $A$ is a finite set. So for any $f\in [A,V]^+$, $\cap_{x\in A} [x,f(x)]^-$ is an open set in $C_h(X)$ containing $f$ and $\cap_{x\in A} [x,f(x)]^-\subseteq [A,V]^+$. Hence $[A,V]^+$ is open in $C_h(X)$.
\end{proof}

\section{Submetrizability and first countability of $C_{kh}(X)$}
In this section, we study the submetrizability and first countability of the space $C_{kh}(X)$. A submetrizable space is one that admits a weaker metrizable topology, or equivalently, one having a continuous injection mapping it into a metric space.

\begin{remark}
\end{remark}
\begin{enumerate}
\item Every pseudocompact set in a submetrizable space is a $G_\delta$-set. In particular all compact subsets, countably compact subsets and the singletons are $G_\delta$-sets in a submetrizable space.
\item If a space X has a $G_\delta$-diagonal, that is, if the set $\{(x, x) : x\in X\}$ is a $G_\delta$-set in the product space $X\times X$, then every point in $X$ is a $G_\delta$-set. Note that
every metrizable space has a $G_\delta$-diagonal. Consequently every submetrizable
space has a $G_\delta$-diagonal. For more details on submetrizable spaces, see \cite{G}.
\end{enumerate}

\begin{theorem}\label{submetrizability of C_kh(X)}
For a space $X$, the following are equivalent.
\begin{enumerate}
\item [$(a)$] $C_{kh}(X)$ is submetrizable.
\item [$(b)$] $C_{kh}(X)$ has a $G_\delta$-diagonal.
\item [$(c)$] Each singleton set in $C(X)$ is a $G_\delta$-set in $C_{kh}(X)$.
\item [$(d)$] $\{0_X\}$ is a $G_\delta$-set in $C_{kh}(X)$.
\item [$(e)$] $X$ is  almost $\sigma$-compact.
\item [$(f)$] $C_k(X)$ is submetrizable.
\end{enumerate}
\end{theorem}
\begin{proof}
\noindent $(a)\Rightarrow (b)\Rightarrow (c)$ and $(e)\Leftrightarrow (f)$. These are well known.

\noindent $(c)\Rightarrow (d)$. It is immediate.

\noindent $(d)\Rightarrow (e)$. Since $0_X\in C(X)$ is a $G_\delta$-set in $C_{kh}(X)$, there exists a countable family $\gamma=\{W_n:n\in \mathbb{N}\}$ of open sets in $C_{kh}(X)$ such that $\{0_X\}=\cap \gamma$. We can assume for each $n\in \mathbb{N}$, $W_n$ is of the form $[A_{n_1},V_{n_1}]^+\cap\ldots\cap[A_{n_{m_n}},V_{n_{m_n}}]^+\cap[U_{n_1},0]^-\cap\ldots\cap[U_{n_{t_n}},0]^-$, where for $1\leq i\leq m_n$, $A_{n_i}$ is compact in $X$ and $V_{n_i}$ is open in $\mathbb{R}$, and for $1\leq j\leq t_n$, $U_{n_j}$ is open in $X$ and $0\in\mathbb{R}$, and $\overline{U_{n_j}}\cap \overline{U_{n_k}}=\emptyset$ for $j\neq k$. For each $W_n\in \gamma$, fix $x_{n_i}\in U_{n_i}$ for $1\leq i\leq t_n$ and put $K_n=\cup_{i=1}^{m_n}A_{n_i}\cup\cup_{j=1}^{t_n}\{x_{n_j}\}$. For each $n\in \mathbb{N}$, $K_n$ is compact.

Let $\mathcal{A}= \cup\{K_n:n\in \mathbb{N}\}$. Suppose that $\overline{\mathcal{A}}\neq X$. So there exists $x_0\in X\setminus \overline{\mathcal{A}}$. Since $X$ is a completely regular space, there exists $f\in C(X)$ such that $f(x_0)=1$ and $f(y)=0$, for all $y\in \overline{\mathcal{A}}$. Therefore $f\in W_n$ for each $n\in \mathbb{N}$. So $f=0_X$, but $f(x_0)=1$. Therefore $\mathcal{A}$ is dense in $X$. Hence $X$ is almost $\sigma$-compact.

\noindent $(f)\Rightarrow (a)$. This follows from $C_k(X)\leq C_{kh}(X)$.
\end{proof}
Our next theorem gives a necessary condition for the first countability of the space $C_{kh}(X)$.
\begin{theorem}
If $C_{kh}(X)$ has a countable local base at the constant zero-function $0_X$, then $X$ is hemicompact.
\end{theorem}
\begin{proof}
Let $\mathcal{W}=\{W_n:n\in \mathbb{N}\}$ be a countable base at $0_X$ in $C_{kh}(X)$. We can assume that for each $n\in \mathbb{N}$, $W_n=[A_{n_1},V_{n_1}]^+\cap\ldots\cap[A_{n_{k_n}},V_{n_{k_n}}]^+\cap[U_{n_1},0]^-\cap\ldots\cap[U_{n_{t_n}},0]^-$. Choose $x_{n_i}\in U_{n_i}$ for $1\leq i\leq t_n$ and put $K_n=\cup_{i=1}^{k_n}A_{n_i}\cup\cup_{j=1}^{t_n}\{x_{n_j}\}$. For each $n\in \mathbb{N}$, $K_n$ is compact.

If $K$ is any compact subspace of $X$, we wish to show that $K\subseteq K_n$ for some $n\in \mathbb{N}$. Now $W=[K,(-1,1)]^+$ is an open set in $C_{kh}(X)$ containing $0_X$. So there exists $W_m\in\mathcal{W}$ such that $W_m\subseteq W$ for some $m\in \mathbb{N}$. Now we show that $K \subseteq K_m$.  Suppose that there exists $x\in K\setminus K_m$. Since $X$ is Tychonoff, there exists $f\in C(X)$ such that $f(x)=1$ and $f(y)=0$ for all $y\in K_m$. Then $f\notin W$ but $f\in W_m$. So our supposition is wrong. Thus $K\subseteq K_m$ and hence $X$ is hemicompact.
\end{proof}

\begin{corollary}\label{Ckh(X) first countable and hemicompact}
If $C_{kh}(X)$ is first countable, then $X$ is hemicompact.
\end{corollary}

\begin{theorem}
If $X$ is a hemicompact metric space, then the constant zero-function $0_X$ has a countable base in $C_{kh}(X)$.
\end{theorem}
\begin{proof}
Let $(K_n)$ be a sequence of compact sets in $X$ such that every compact set in $X$ is contained in $K_m$ for some $m\in \mathbb{N}$. Let $\mathcal{K}=\{K_n:n\in \mathbb{N}\}$. Since $X$ is a hemicompact metric space, $X$ is second countable. Let $\mathcal{B}$ be a countable base for $X$ and $\mathcal{V}$ be a countable local base at $0\in \mathbb{R}$. Now consider the collection $\mathcal{F}$ of open sets in $C_{kh}(X)$ containing $0_X$ of the form $$[K_{n_1},V_{n_1}]^+\cap\ldots\cap[K_{n_k},V_{n_k}]^+\cap[U_{n_1},0]^-\cap\ldots\cap[U_{n_t},0]^-,$$ where $K_i\in\mathcal{K}$, $V_i\in\mathcal{V}$ and $U_j\in \mathcal{B}$ for each $i\in \{n_1,\ldots, n_k\}$ and $j\in \{n_1,\ldots, n_t\}$. Clearly $\mathcal{F}$ is a countable family. Now we show that $\mathcal{F}$ forms a base at $0_X$.

Let $G=[A_1,W_1]^+\cap\ldots\cap[A_m,W_m]^+\cap[U_1,0]^-\cap\ldots\cap[U_n,0]^-$ be any basic open set in $C_{kh}(X)$ containing $0_X$, where $A_i$ is a compact subset of $X$, $W_i$ is open in $\mathbb{R}$ and $U_j\in \mathcal{B}$, $0\in \mathbb{R}$ for each $i\in \{1,\ldots,m\}$, $j\in \{1,\ldots,n\}$ and $\overline{U_j}\cap \overline{U_k}=\emptyset$ for $j\neq k$. Since $0_X\in G$, for each $i\in \{1,\ldots,m\}$, we have $0\in W_i$. Since $X$ is hemicompact and $\mathcal{V}$ is a countable local base at $0\in \mathbb{R}$, for each $1\leq i\leq m$, there exist $K_{k_i}\in \mathcal{K}$ and $V_{k_i}\in \mathcal{V}$ such that $A_i\subseteq K_{k_i}$ and $V_{k_i}\subseteq W_i$. Clearly $0_X\in [K_{k_1},V_{k_1}]^+\cap\ldots\cap[K_{k_m},V_{k_m}]^+\cap[U_1,0]^-\cap\ldots\cap[U_n,0]^-\subseteq G$.
\end{proof}
Now Corollary \ref{Ckh(X) first countable and hemicompact} implies that $X$ being hemicompact is a necessary condition for the first countability of the space $C_{kh}(X)$. In our next theorem, we show that $X$ being hemicompact is also a sufficient condition for the first countability of the space $C_{kh}(X)$ whenever $X$ is a locally connected metric space.

\begin{theorem}
If $X$ is a locally connected metric space, then the following are equivalent.
\begin{enumerate}
  \item [$(a)$] $X$ is a hemicompact space.
  \item [$(b)$] $C_{kh}(X)$ is first countable.
\end{enumerate}
\end{theorem}
\begin{proof}
\noindent $(a)\Rightarrow (b)$. Since $X$ is hemicompact, $C_k(X)$ is first countable. Now the proof follows from Corollary 4.14 in \cite{AMK2}.

\noindent $(b)\Rightarrow (a)$. This follows from Corollary \ref{Ckh(X) first countable and hemicompact}.
\end{proof}
By Corollary \ref{Ckh(X) first countable and hemicompact}, hemicompactness of $X$ is a necessary condition for the metrizability of the space $C_{kh}(X)$. Our next example shows that even for a locally connected metrizable space $X$, hemicompactness of $X$ is not a sufficient condition for the metrizability of $C_{kh}(X)$.

\begin{example}
Let $X$ be the space of real numbers with usual topology. Then by Theorem \ref{R-separability of Ckh(X)} of next section, we have $C_{kh}(X)$ is separable. But $C_{kh}(X)$ is not second countable for any topological space $X$ by Theorem \ref{second_countability_of_Ckh(X)} of the next section. So it cannot be metrizable.
\end{example}

\section{separability of $C_{kh}(X)$}
In this section, we discuss the separability of the space $C_{kh}(X)$. For this we first need to study the $R$-sets and $R$-separable spaces.

 A nonempty subset $B$ of a space $X$ is said to be an $R$-set if there exists a countable collection $T=\{f_n : n\in \mathbb{N}\}$ of real-valued continuous functions on $X$ such that $\bigcup_{n\in \mathbb{N}}f_n(B)=\mathbb{R}$. A space $X$ is said to be $R$-separable if there exists a countable collection of $R$-sets in $X$ such that every nonempty open set in $X$ contains some member of this collection.

It is easy to see that every $R$-set is uncountable and any set containing an $R$-set is an $R$-set. Every $R$-separable space is separable and contains no isolated point. In order to show that there exists a large class of $R$-separable spaces, we prove the following propositions.

\begin{proposition}
A perfect Polish space is $R$-separable.
\end{proposition}
\begin{proof}
Let $X$ be a perfect Polish space and $\mathcal{B}=\{U_n : n\in \mathbb{N}\}$ be a countable base for $X$. It is sufficient to prove that there is a countable collection $\mathcal{S}$ of $R$-sets such that every element of $\mathcal{B}$ contains an element of $\mathcal{S}$. By using the regularity of $X$, for each $n\in \mathbb{N}$, if $x\in U_n$ there exists $V_n$ open in $X$ such that $x\in V_n\subseteq \overline{V_n}\subseteq U_n$. As $X$ is a perfect space and $V_n$ is open in $X$, this implies that $V_n$ contains no isolated point. Then $\overline{V_n}$ is a perfect set in $X$ (see exercise $30B.3$ of \cite{sw}). Since closed subset of a Polish space is a Polish space and every perfect Polish space contains a Cantor set, $\overline{V_n}$ is a perfect Polish space and it contains a Cantor set $S_n$. Next, we show that $S_n$ is an $R$-set. Since $\mathbb{R}=\cup[-m,m]$ and every compact metric space is a continuous image of a Cantor set, for each $m\in \mathbb{N}$, there exists $f_m\in C(X)$ such that $f_m(S_n)=[-m,m]$. Thus $\cup_{m\in \mathbb{N}} f_m(S_n)=\mathbb{R}$. Hence $S_n$ is an $R$-set.
\end{proof}

\begin{proposition}\label{corollarylocallyconnectedandRseparable}
If $X$ is a space having a countable $\pi$-base consisting of nontrivial connected sets, then $X$ is $R$-separable.
\end{proposition}
\begin{proof}
It is sufficient to prove that every nontrivial connected subset of $X$ is an $R$-set. Let $U$ be a nontrivial connected subset of $X$, and let $x_U$, $y_U$ be distinct elements of $U$. Since $X$ is a Tychonoff space and $U$ is a connected subset of $X$, for each $n\in \mathbb{N}$, there exists $f_n\in C(X)$ such that $[-n,n]\subseteq f_n(U)$. Therefore $\cup_{n\in \mathbb{N}} f_n(U)=\mathbb{R}$, and hence $U$ is an $R$-set.
\end{proof}

We have the following necessary and sufficient conditions for the separability of the space $C_{kh}(X)$.
\begin{theorem}\label{R-separability of Ckh(X)}
Let $X$ be a space with a countable $\pi$-base. Then the following are equivalent.
\begin{enumerate}
\item [$(a)$] $C_{kh}(X)$ is separable.
\item [$(b)$] $C_{ph}(X)$ is separable.
\item [$(c)$] $X$ is an $R$-separable submetrizable space.
\end{enumerate}
\end{theorem}
\begin{proof}
\noindent $(a)\Rightarrow (b)$. It follows from the fact that $C_{ph}(X)\leq C_{kh}(X)$.

\noindent $(b)\Rightarrow (c)$. If $C_{ph}(X)$ is separable, then $C_p(X)$ is separable. Therefore, Corollary 4.2.2 in \cite{mn} implies that $X$ is submetrizable. Now we show that every open set in $X$ is an $R$-set. Let $T=\{f_n : n\in \mathbb{N}\}$ be a countable dense set in $C_{ph}(X)$ and $U$ be any nonempty open set in $X$. Suppose that there exists a $y\in \mathbb{R}\setminus \bigcup_{n\in \mathbb{N}}f_n(U)$. Then $[U, y]^-$ is a nonempty open set in $C_{ph}(X)$ which does not intersect with $T$. Thus $U$ is an $R$-set in $X$. Since $X$ has a countable $\pi$-base, $X$ is $R$-separable.

\noindent $(c)\Rightarrow (a)$. Suppose that $\mathcal{K}=\{S_n: n\in\mathbb{N}\}$ is a countable family of $R$-sets in $X$ such that every nonempty open subset of $X$ contains some member of $\mathcal{K}$.  Let $\mathcal{U}$ be a countable $\pi$-base for $X$ and $\mathcal{V}$ be a countable base for $\mathbb{R}$ consisting of bounded open intervals. For convenience of notation, for each
$n\in \mathbb{N}$, let $\hat{\mathcal{U}}^n$ be the set of $(U_1,\ldots,U_n)\in \mathcal{U}^n$ such that $\{U_1,\ldots,U_n\}$ is a collection of pairwise disjoint sets. Since $X$ is a separable submetrizable space, the space $C_k(X)$ is separable. So let $\mathcal{F}$ be a countable dense subset of $C_k(X)$.

For each $f\in \mathcal{F}$ and $n\in \mathbb{N}$, let $\mathcal{T}_f^n=\{((U_1,\ldots,U_n),(V_1,\ldots,V_n))\in\hat{\mathcal{U}}^n \times \mathcal{V}^n: f(\overline{U_i})\subseteq V_i$ for $1\leq i\leq n\}$. Since $f$ is continuous and $X$ is a Tychonoff space, for each $n\in\mathbb{N}$, $\mathcal{T}_f^n\neq \emptyset$. Then $\mathcal{T}_f=\cup_{n\in \mathbb{N}}\mathcal{T}_f^n$ is countable.

Let $T\in \mathcal{T}_f$, say $T=((U_1,\ldots,U_n),(V_1,\ldots,V_n))$. Since $X$ is an $R$-separable Tychonoff space, for each $i=1,\ldots,n$, there exist $E_i\in \mathcal{U}$, $S_i\in \mathcal{K}$ and a countable collection $F_{S_i}$ in $C(X)$ such that $S_i\subseteq E_i\subseteq U_i$, $\overline{U_i}\setminus E_i$ and $S_i$ are completely separated sets in $X$ and $\cup \{g(S_i):g\in F_{S_i}\}=\mathbb{R}$. For each $i=1,\ldots,n$, take $\mathcal{L}_i=\{f^{S_{i}}\in F_{S_{i}}:(f^{S_{i}})^{-1}(V_i)\cap S_{i}\neq\emptyset\}$; then for each $f^{S_{i}}\in\mathcal{L}_i$, there exists a continuous function $f^{S_{i}}_i:X \rightarrow \overline{V_i}$ such that $f^{S_{i}}_i(x)=f^{S_{i}}(x)$ for  $x\in (f^{S_{i}})^{-1}(V_i)\cap S_{i}$ and $f^{S_{i}}_i(x)=f(x)$ for each $x\in \overline{U_i}\setminus E_i$. Take $S_T=\{S_{1},\ldots,S_{n}\}$ and let $\mathcal{H}_{S_T}=\{ (f^{S_{1}}_1,\ldots,f^{S_{n}}_n): f^{S_{i}}\in \mathcal{L}_i, 1\leq i\leq n\}$. Clearly $\mathcal{H}_{S_T}$ is countable.

Then for each $T\in \mathcal{T}_f$ and $H\in\mathcal{H}_{S_T}$, where $$T=((U_1,\ldots,U_m),(V_1,\ldots,V_m)) \text{ and } H = \{f^{S_1}_1,\ldots,f^{S_m}_m\}$$ define a continuous function $f_{T,H}:X \rightarrow \mathbb{R}$ as follows.

\[f_{T,H}(x) =\left\{\begin{array}{cc}
                  f^{S_i}_i(x) & x\in U_i \mbox{ and } 1\leq i\leq m \\
                  f(x) & x\in X\setminus\ U_i\mbox{ and } 1\leq i\leq m.\\
                \end{array}\right.
   \]
Consider $\mathcal{F'}$ to be the collection $\{f_{T,H}:f\in \mathcal{F}, T\in \mathcal{T}_f, H\in\mathcal{H}_{S_T}\}$, which we now prove is dense in $C_{kh}(X)$.

Let $G$ be any nonempty open set in $C_{kh}(X)$, then Proposition \ref{pibase for Ckh(X)} implies that $G$ contains a nonempty open set $B=[A_1,G_1]^+\cap\ldots\cap[A_m,G_m]^+\cap [U_1, r_1]^-\cap\ldots\cap[U_n, r_n]^-$ such that for each $j = 1,\ldots,n$, either $U_j\cap A_i=\emptyset$ for all $i=1,\ldots,m$, or $U_j\subseteq\cup\{A_i : r_j\in G_i\}$ and $U_j\cap \cup\{A_i : r_j \notin G_i\}=\emptyset$; and $A_i$ is a compact subset of $X$, $G_i\in \mathcal{V}$ whenever $i\in\{1,\ldots,m\}$, and $\overline{U_i}\cap\overline{U_j}=\emptyset$ for $i\neq j$. Let $f\in \mathcal{F}\cap [A_1,G_1]^+\cap\ldots\cap[A_m,G_m]^+$ and let $I_j=\{1\leq i\leq m:r_j\in G_i, U_j\cap A_i\neq\emptyset\}$. We can assume that $I_j\neq\emptyset$ and $|I_j|=t_j$ for $1\leq j\leq k$, where $k\leq n$ and $I_j=\emptyset$ for $k+1\leq j\leq n$.

If $I_j\neq\emptyset$, then for each $i\in I_j$, $r_j\in G_i$ and $U_j\cap A_i\cap f^{-1}(G_i)\neq\emptyset$. Since $X$ is a perfect space, for each $i\in I_j$, we can choose $x_i^j\in f^{-1}(G_i)\cap U_j$ such that $\{x_i^j:i\in I_j\}$ is a collection of distinct points. Also there exists $V_i^j\in \mathcal{V}$ such that $f(x_i^j), r_j\in V_i^j\subseteq \overline{V_i^j}\subseteq G_i$. Since $X$ is a Hausdorff space, there exists $\{B_i^j:i\in I_j\}$ a collection of disjoint open sets in $X$ such that $x_i^j\in B_i^j\subseteq \overline{B_i^j}\subseteq f^{-1}(V_i^j)\cap U_j$. Then this implies that there exists $(U_{1}^j,\ldots,U_{t_j}^j)\in \hat{\mathcal{U}}^{t_j}$ such that $U_i^j\subseteq \overline{U_i^j}\subseteq \overline{B_i^j}\subseteq  f^{-1}(V_i^j)\cap U_j$ for each $i\in I_j$. Now for each $i\in I_j$, $f(\overline{U_i^j})\subseteq V_i^j$. Also there exist $S_i^j\in\mathcal{K}$, $f^{S_i^j}\in F_{S_i^j}$ and $E_i^j\in \mathcal{U}$ such that $S_i^j\subseteq E_i^j\subseteq U_i^j$, $\overline{U_i^j}\setminus E_i^j$ and $S_i^j$ are completely separated in $X$ and $r_j\in f^{S_i^j}(S_i^j)$ for each $i\in I_j$.

If $I_j=\emptyset$, then $U_j\cap A_i=\emptyset$ for all $i\in\{1,\ldots,m\}$. Now take any $x\in U_j$ and let $V_j\in\mathcal{V}$ be any neighborhood of $f(x)$ containing $r_j$. Since $f$ is continuous, there exists $U_j^{'}\in \mathcal{U}$ such that $ U_j^{'}\subseteq\overline{U_j^{'}}\subseteq U_j$ and $f(\overline{U_j^{'}})\subseteq V_j$. Also there exist $S_j\in\mathcal{K}$, $f^{S_j}\in F_{S_j}$ and $E_j^{'}\in\mathcal{U}$ such that $S_j\subseteq E_j^{'}\subseteq U_j^{'}$, $\overline{U_j^{'}}\setminus E_j^{'}$ and $S_j$ are completely separated in $X$ and $r_j\in f^{S_j}(S_j)$. Now take

$$U=(U_1^1,\ldots,U_{t_1}^1,U_{1}^2,\ldots,U_{t_2}^2,\ldots, U_{1}^k,\ldots,U_{t_k}^k,U_{k+1}^{'},\ldots,U_{n}^{'}),$$ $$V=(V_1^1,\ldots,V_{t_1}^1,V_1^2,\ldots,V_{t_2}^2,\ldots, V_{1}^k,\ldots,V_{t_k}^k,V_{k+1},\ldots,V_{n}).$$
 Since $\overline{U_i}\cap \overline{U_j}=\emptyset$ for $i\neq j$, we have $T=(U,V)\in \mathcal{T}_f$. Take $$H=(f^{S_1^1}_1,\ldots,f^{S_{t_1}^1}_{t_1},f^{S_1^2}_1,\ldots,f^{S_{t_2}^2}_{t_2},\ldots, f^{S_1^k}_1,\ldots,f^{S_{t_k}^k}_{t_k},f^{S_{k+1}}_{k+1},\ldots,f^{S_{n}}_n).$$
     Then it is easy to see $f_{T,H}\in \mathcal{F'}\cap B$.
\end{proof}

Since the sum function $s:C_{h}(\oplus \{X_\alpha :\alpha \in \Gamma\})\rightarrow \Pi\{C_{h}(X_\alpha): \alpha\in \Gamma\}$ is a homeomorphism (see Proposition 2.11 of \cite{AMK2}), we have the following result.
\begin{corollary}
If $X$ is the topological sum of c (where $c=2^{\aleph_0}$) or fewer $R$-separable metrizable spaces, then $C_{kh}(X)$ is separable.
\end{corollary}

\begin{corollary} \label{locallyconnected_Ckh(X) separability}
If $X$ is a locally connected separable metric space without isolated points, then the space $C_{kh}(X)$ is separable.
\end{corollary}

\begin{corollary}\label{irrationalCkh(X)separability}
If $X$ is a perfect Polish space, then the space $C_{kh}(X)$ is separable.
\end{corollary}
\begin{example}
Let $X=\mathbb{P}$ be the space of irrationals. Then Corollary \ref{irrationalCkh(X)separability} implies that the space $C_{kh}(X)$ is separable.
\end{example}
\begin{example}
Let $X=S$ be the Sorgenfrey line. Then Theorem \ref{R-separability of Ckh(X)} implies that the space $C_{kh}(X)$ is separable.
\end{example}
Finally, the last theorem, whose proof is similar to the proof of Theorem 5.12 in \cite{AMK}, points out that there are certain countability properties that the space $C_{kh}(X)$ never has.

\begin{theorem}\label{second_countability_of_Ckh(X)}The space $C_{kh}(X)$ is neither Lindel$\ddot{o}$f nor second countable.
\end{theorem}

\end{document}